\documentclass[a4paper,11pt]{article}
\usepackage[utf8]{inputenc}
\usepackage{amsmath, amsfonts, amsthm,amssymb,  bbm}
\usepackage{math}
\usepackage{hyperref, geometry} 
\usepackage[usenames, dvipsnames]{color}
\usepackage{xcolor}
\newtheorem{lemma}{Lemma}[section]
\newtheorem{theorem}[lemma]{Theorem}
\newtheorem{proposition}[lemma]{Proposition}
\newtheorem{corollary}[lemma]{Corollary}
\newtheorem{remark}[lemma]{Remark}

\newtheorem{definition}[lemma]{Definition}

 \usepackage{mathrsfs}

\def\beq{\begin{equation}}   \def\eeq{\end{equation}}
\def\bea{\begin{eqnarray}}  \def\eea{\end{eqnarray}}

\newcommand{\dd}{  \text{d}   }

\newcommand{\eps}{\varepsilon}

\def\norma#1{\left\|#1\right\|}

\def\csi{\xi}
\def\heta{\eta}

\newcommand{\om}{{\omega}}

\newcommand{\p}{  \partial   }
\renewcommand{\s}{  \sigma   }
\newcommand{\ka}{  \kappa   }
\newcommand{\meas}{\operatorname{meas}}
\renewcommand{\be}{\begin{equation}}
\newcommand{\ee}{\end{equation}}

\numberwithin{equation}{section}

\renewcommand{\bar}{\overline}

\title{Longtime dynamics for the Landau Hamiltonian  with  a time dependent magnetic field }

\author{
D. Bambusi\footnote{Dipartimento di Matematica Federigo Enriques, Universit\`a degli Studi di Milano, Via Saldini 50, I-20133
Milano. \newline
 \textit{Email: } \texttt{dario.bambusi@unimi.it}},
B. Gr\'ebert\footnote{Laboratoire de Math\'ematiques Jean Leray, Universit\'e de Nantes, 2 rue de la Houssini\`ere
BP 92208, 44322 Nantes Cedex 3 \newline
 \textit{Email: } \texttt{benoit.grebert@univ-nantes.fr}} ,
A. Maspero\footnote{ International School for Advanced Studies (SISSA), Via Bonomea 265, 34136, Trieste, Italy \newline
 \textit{Email: } \texttt{amaspero@sissa.it}}, 
  D. Robert\footnote{Laboratoire de Math\'ematiques Jean Leray, Universit\'e de Nantes, 2 rue de la Houssini\`ere
BP 92208, 44322 Nantes Cedex 3 \newline
 \textit{Email: } \texttt{didier.robert@univ-nantes.fr}},
 C. Villegas-Blas\footnote{Universidad Nacional Autonoma de M\'exico, Instituto de Matem\'aticas, Unidad Cuernavaca \newline
 \textit{Email: } \texttt{villegas@matcuer.unam.mx}}
}

\date{}

\begin{document}

	\maketitle

\begin{abstract}
We consider a modulated magnetic field, $B(t) = B_0 +\eps f(\omega t)$,  perpendicular to a fixed plane, where $B_0$ is constant, $\varepsilon>0$  and $f$ a periodic function on the torus ${\mathbb T}^n$.
 Our aim is to study  classical and  quantum dynamics for the corresponding Landau Hamiltonian. It turns out that the  results depend strongly on the chosen  gauge. For the Landau gauge the position  observable 
 is unbounded for "almost all" non resonant frequencies $\omega$. On the contrary, for the symmetric gauge we obtain that, for "almost all" non resonant frequencies $\omega$, the Landau Hamiltonian is reducible to a two dimensional harmonic oscillator and thus gives rise to bounded dynamics. The proofs use KAM algorithms  for the classical dynamics. Quantum applications are given. In particular, the Floquet spectrum is absolutely continuous in the Landau gauge while it is discrete, of finite multiplicity, in symmetric gauge.
\end{abstract}
	
\begin{flushright}
\textit{
Thank you, Thomas, for sharing your enthusiasm \\
and your joy of playing with mathematics.}
 \end{flushright}

\section{Introduction and main results} 

In this paper we study the dynamics of time dependent perturbations of the Schr\"odinger  equation
\begin{equation}
\label{LS}
\im\partial_t\psi = H_{A^\#}(t) \psi + V(t) \psi \ , 
\end{equation}
where  
$H_{A^\#}(t)$ is the magnetic Schr\"odinger operator in $L^2(\R^3)$:
\begin{equation}
\label{landau1}
H_{A^\#}(t) := \sum_{1\leq j\leq 3}\left(D_{x_j}-A^{\#}_j(t,x)\right)^2,\;\; D_{x} := \im^{-1}\frac{\partial}{\partial x} \ , 
\end{equation}
with $A^{\#}(t,x) =(A^{\#}_1(t,x), A^{\#}_2(t,x),A^{\#}_3(t,x))$  a time dependent {vector} potential 
and finally 
 $V(t,x)$ is a time dependent {scalar} potential.
 We recall that the 
  electric field is given by
   $\vec E(t, x)=-\frac{\partial A^{\#}}{\partial t}(t,x)-\nabla_x V(t,x)$ and the magnetic field  by 
$\vec B(t,x) = \nabla_x\wedge A^{\#}(t, x)$. We shall assume that the magnetic field has a fixed direction orthogonal to the plane $\{e_1, e_2\}$.
Choosing $A^{\#}_3(t,x)\equiv 0$ then $A_1^{\#}(t, x)$ and  $A_2^{\#}(t, x)$ depend only on $(t,x_1,x_2)$ and it is enough to consider the 2D  magnetic Hamiltonian, {with a new simpler notation}, 
 $$
 H_{A^{\#}}(t) = \left(D_{x_1}-A_1^{\#}(t,x)\right)^2 + \left(D_{x_2}-A_2^{\#}(t,x)\right)^2
 $$ as an operator in  $L^2(\R^2)$. \\ 
 An important particular case is  the  constant (in position)  magnetic  field ${\mathbf B}(t) = (0,0, -B(t))$, for which  
 \begin{equation}\label{AB} B(t) = \partial_{x_1} A_2^{\#} - \partial_{x_2}
   A_1^{\#}   \ . \end{equation}
{This is usually studied using either the symmetric Gauge or the
  Landau Gauge, namely}
 \begin{itemize}
 \item[-]{\bf the symmetric gauge}:  $A_1^{\#}(t,x) = \frac{B(t)}{2}x_2,  \ \  A_2^{\#}(t,x) = -\frac{B(t)}{2}x_1$;
 \item[-] {\bf the Landau gauge:} $A_1^{\#}(t,x) = B(t)x_2,  \ \  A_2^{\#}(t,x) = 0$.
 \end{itemize}
 In this paper we consider the case when $B$ slightly fluctuates
 around a fix value $B_0>0$: $B(t) = B_0 +\eps f(\omega t)$ where
 $\omega\in\R^n$ is a frequency vector, $f$ is a periodic function
 real analytic on the torus ${\mathbb T}^n$ and $\varepsilon>0$ is a
 small parameter.

 Mathematically the source of most of the interesting features of
   the Landau Hamiltonian rests in the fact that when $\eps=0$ the
   Hamiltonian is degenerate, in the sense that it is equivalent
   (unitary equivalent in the quantum case, canonically equivalent in
   the classical case) to the Hamiltonian of a 1-d harmonic
   oscillator. As a result the quantum spectrum of the system is
   composed just by essential spectrum and coincides with the set
   $\left\{\lambda_j=2B_0(j+1/2)\mid {j\in\N}\right\}$.

 The case $\eps\not=0$  will be discussed in the two  different gauges: \\
 (i) the Landau gauge
$H_L(t) = (D_{x_1}-B(t)x_2)^2 +D_{x_2}^2$; \\
(ii) the symmetric gauge $H_{sL}(t) = (D_{x_1}-\frac{B(t)x_2}{2})^2
 +(D_{x_2}+\frac{B(t)x_1}{2})^2$.
\\
{Notice that, for $\varepsilon =0$,  $B(t)=B_0$ hence $H_L$ and $H_{sL}$ are gauge equivalent,  but for $\epsilon\neq 0$ this equivalence is broken.}

 It turns out that both the main part of the
  Hamiltonian and the time dependent perturbation are 
  quadratic polynomials in the position and the momentum variables,
  and this allows to study the problems (both classical and quantum)
  using the ideas of \cite{BGMR1}, namely by using classical KAM
  theory to conjugate the Hamiltonian to a suitable normal form whose
  dynamics is easy to study. \\
  The results depend drastically of the choice of the gauge:
\begin{itemize}
\item In  case (i), provided $\omega$ is non resonant, a condition
  which is fulfilled in a set of asymptotically full measure, we get
  that for $\eps\not=0$ the position observable is unbounded as
  $t\to\infty$ as well for the classical motion and the quantum
  motion. {\em It may be surprising that for a dynamical system a
    non-resonance condition generates an instability}\footnote{But the
      phenomenon is similar of that encoded in Theorem 3.3 of
      \cite{BGMR1}, in which the non resonance condition is used to
      eliminate from the Hamiltonian as many terms as possible, so
      that one remains only with the terms actually generating the
      instability.}.\\
      As a consequence, in the quantum side, we prove that the Floquet spectrum is absolutely continuous.
\item
In case (ii) we prove that for $\omega$ in a set of
  asymptotically full measure,  the dynamics is reducible to a
  harmonic oscillator with two degrees of freedom, hence with bounded
  dynamics. \\
  As a consequence, in the quantum side, we prove that the Floquet spectrum is discrete with finite multiplicity.
  \end{itemize}

Notice that $H_L(t)$ and $H_{sL}(t)$ are gauge equivalent modulo a
quadratic scalar potential (see section \ref{changegauge}). Therefore
the two models are not physically equivalent: in the two cases we have
the same magnetic field but not the same electric field.\\

\vskip20pt

\subsection{Main result in the Landau gauge}\label{landgauge}



We consider first the  Landau gauge, namely
\begin{equation}
\label{landauB}
H_L(t) = \left(D_{x_1} - B(t)x_2\right)^2   + D_{x_2}^2 \ , \qquad B(t) >0 \ , 
\end{equation}
with  $B(t):=B_0+\eps f(\omega t)$, $f$ is real analytic on the torus
${\mathbb T}^n$, $\hat f(0)=0$
and $\omega\in [0,2\pi)^n:=\D$. Here and below we denote by $\hat f(k)$ the $k$-th Fourier coefficients of $f$, 
    $$
\hat f(k):=(2\pi)^{-n}\int_{\T^n}f(\theta)e^{-\im k\cdot \theta}d\theta\ .
    $$
We decompose the Hamiltonian $H_L(t)$ in \eqref{landauB} as 
$$H_L(t)= H_L+ R_L(\omega t)$$
where 
\begin{align*}
H_L &= \left(D_{x_1} - B_0 x_2\right)^2   + D_{x_2}^2,\\
R_L(\omega t)&= -2\eps f(\omega t) x_2 (D_{x_1}-B_0 x_2)+\eps^2f(\omega t)^2 x_2^2.
\end{align*}
We denote by $h_L(t)$ the corresponding classical Hamiltonian:
\begin{equation}
\label{landaub}h_L(t,x,p)= (p_1-B(t)x_2)^2+p_2^2=h_L(x,p)+r_L(\omega t,x,p).\end{equation}
We introduce now complex coordinates in which  the classical Hamiltonian $h_L$ has the form of a degenerate 2-dimensional Harmonic oscillator. 
First introduce 
 the symplectic variables
$$ Q_1=\frac{-1}{B_0}(p_1-B_0x_2),\quad P_1=p_2$$
and
$$Q_2=\frac{-1}{B_0}(p_2-B_0x_1),\quad P_2=p_1.$$
In these variables we have
$$h_L=B_0^2Q_1^2+P_1^2.$$
Then we introduce the complex variables
$$z_1=\frac{B_0Q_1+\im P_1}{\sqrt{2B_0}}, \quad z_2=\frac{B_0Q_2+\im P_2} {\sqrt{2B_0}} \ , 
$$
fulfilling  $\im dz_i\wedge d\bar z_i= dQ_i\wedge dP_i$, $i=1,2$. 
In these variables
\begin{equation}\label{hL}
h_L=2 B_0 |z_1|^2.
\end{equation}
 The link with the initial coordinates $(x,p)\in\R^4$ is given by $(z_1,z_2) = \tau_0(x,p)$ where $\tau_0$ is the  linear symplectic   transformation 
$\tau_0: \R^4\rightarrow\C^2$ such that 
\bea\label{coord1}
z_1 &= &\frac{B_0}{\sqrt{2B_0}}\left(\frac{B_0x_2 -p_1}{B_0}\right) + \im \frac{p_2}{\sqrt{2B_0}} \nonumber\\
z_2 &= &\frac{B_0}{\sqrt{2B_0}}\left(\frac{B_0x_1 -p_2}{B_0}\right) + \im \frac{p_1}{\sqrt{2B_0}}.
\eea

In order to state the first result we need to define a constant
  $c_\omega$ which is only defined for $\omega$ fulfilling a
  nonresonance condition. To this end we preliminary restrict the set
  of the allowed frequencies.
  \begin{definition}
    \label{diof}
The set $\D_0\subset[0,2\pi]^n$ is the set of the frequencies $\omega$
s.t. there exist positive $\gamma$, $\tau$ s.t.
\begin{align}
   \label{diof.1}
\left|\omega\cdot k+2B_0\right|\geq\frac{\gamma}{1+|k|^\tau}\ ,\quad
\forall k\in\Z^n
\\
\label{diof.2}
\left|\omega\cdot k\right|\geq\frac{\gamma}{|k|^\tau}\ ,\quad
\forall k\in\Z^n\setminus\left\{ 0\right\}\ .
\end{align}
\end{definition}
Remark that such a set has full measure in $[0,2\pi]^n$.

We now introduce 
\be\label{c}
c_\omega:= \int_{\T^n} {g_\omega(\theta)^2 } d\theta
\ee
where the function 
$$
g_\omega(\theta) := -\frac1{\sqrt{2B_0}}\sum_{k\neq 0}\frac{\omega\cdot
  k}{\omega\cdot k+2 B_0}\hat f(k) e^{\im k\cdot \theta},
  $$
is  well defined for $\omega\in\D_0$ (recall that $f$ is real analytic).

Furthermore, in the following we will say that two polynomials are
$O(\epsilon)$ close from each other, if their coefficients are
$O(\epsilon)$ close from each other.
 
Our first result is the following:
\begin{theorem}
  \label{len.mag} There exists $\eps_0>0$ such that for $|\eps|<\eps_0$, 
   there exist
   \begin{itemize}
   \item an asymptotically full measure set of frequencies $\mathcal C_\eps\subset \D_0$ satisfying:
   $$\lim_{\eps\to 0}{\text{meas}(\D_0\setminus\mathcal C_\eps)}=0;$$
   \item  a linear, symplectic change of variable $\tau$,   depending on $\omega, t,\eps$,  which is close to the identity, namely  $\tau = \mathbb{I} + O(\eps)$
   uniform in the other parameters,
   \end{itemize}
    such that, for $\omega\in\mathcal C_\eps$, the time quasi-periodic Hamiltonian $h_L(t,x,p)$ in \eqref{landaub} is conjugated to the constant coefficient quadratic Hamiltonian
    \begin{equation}
      \label{bc.1}
(h_L(t)\circ \tau_0^{-1})\circ \tau =b(\eps)|z_1|^2+c(\eps)(z_2-\bar z_2)^2
      \end{equation}
where 
\be\label{bc}b(\eps)= 2B_0+O(\eps^2),\quad c(\eps)= c_\omega\eps^2+O(\eps^4)\ee
and $ c_\omega$ is given by \eqref{c}.
\end{theorem}
In the new  coordinates $(z_1, z_2)$ the motion is easily computed:
\bea
z_1(t) &=& {\rm e}^{-\im b(\eps)t}z_1(0) \ , \nonumber\\
\Im z_2(t)  &=&\Im z_2(0),\;\; \Re z_2(t) = -4c(\eps)\Im z_2(0)t  + \Re z_2(0) \ .\nonumber
\eea
   Let us come back to the original coordinates  $(x,p)$.  First
 we remark that the Hamiltonian at r.h.s. of \eqref{bc.1} takes the
 form
 $$
\frac{b(\eps)}{2B_0}h_L(x,p)+\alpha c(\varepsilon)p_1^2
 $$
with $h_L$ the original Landau Hamiltonian and $\alpha\not=0$ a
numerical constant. As a consequence, the corresponding dynamics is
just given by the standard circular motion of a particle in a magnetic
field with a slightly
different frequency and with superimposed a uniform motion in the
direction of $x_1$. This uniform motion is the new effect which gives
rise to the growth of the solution. Actually this description holds in
 the system of coordinates introduced by the KAM procedure. In the
 true original coordinates this motion is slightly deformed, so that
 it has superimposed a small oscillation. Precisely
  the motion for the quadratic Hamiltonian $h_L(t)$  is a linear flow
 $\Phi_B(t)$   where we have 
 $$
 \begin{pmatrix} x(t)\\p(t)\end{pmatrix} = \Phi_B(t) \begin{pmatrix}x(0)\\p(0)\end{pmatrix}.
 $$
where $\Phi_B(t)$ is a real $4\times 4$ symplectic matrix 
{(it is the classical Hamiltonian flow of the classical Hamiltonian $h_L(t)$)}.
Then we have:

\begin{corollary}\label{class_esc}
For $\omega\in\mathcal C_\eps$  we have, with $\alpha\neq 0$, 
\beq\label{1nb}
x_1(t)=x_1(0) +\alpha c(\eps)p_1(0) t +\frac{1}{B(t)}(p_2(t)-p_2(0)) + a_{\varepsilon,\omega}(t)\cdot x(0) + b_{\varepsilon,\omega}(t)\cdot p(0)
\eeq
{where $\vert a_{\varepsilon,\omega}(t)\vert + \vert b_{\varepsilon,\omega}(t)\vert = O(1)$ for $0<\varepsilon <1$ and $t\in\R$}.\\

Moreover,  modulo an error term of the form $E_{\tilde a,\tilde b}(t) = \tilde a(t)\cdot x(0) + \tilde b(t) \cdot p(0) $,
 such that  uniformly for $t\in\R$, $\omega\in\mathcal C_\eps$, we have $\vert\tilde a(t)\vert  +\vert\tilde b(t)\vert =O(\eps)$,

 \bea
p_2(t) &=& \sqrt{2B_0}\Im(z_1(t)),  \;\;z_1(t) = {\rm e}^{-2ibt}z_1(0) \\\nonumber
p_1(t)  &=&p_1(0) \nonumber\\
x_2(t) &=&  \frac{p_1(0)}{B(t)} + \frac{\sqrt{2}B_0^{3/2}}{B(t)^2}\Re z_1(t).
\eea
\end{corollary}
In particular,  if  both $c_\omega$ and $p_1(0)$ are not zero, then the classical flow is not bounded as soon as $\eps\neq0$ is small enough.
\begin{remark} {Clearly, in view
  of \eqref{c}, $c_\omega\neq 0$ holds for  {$\omega$ in a set of
  asymptotically full
  measure and for  $f$ in a set of codimension 1 (e.g. in $L^2$)}. For instance,
by a simple calculus one has that if   $f(\theta)=\sin (\theta)$ {(and thus $n=1$)} then $c_\omega\neq 0$ as soon as $\omega\neq 2B_0$.}
\end{remark}
\smallskip
From the  result {on the classical evolution of }$x_1(t)$, we get 
 a direct application to the large time evolution of the quantum position observable $\hat x_1(t)$. Let us explicit our notations: we denote by $\hat x_j$ and $\hat p_j$, $j=1,2$, the position and momentum operator
 $$
 \hat x_j\psi(x) =x_j\psi(x), \; \hat p_j\psi=\frac{1}{\im}\frac{\partial}{\partial x_j}\psi, \;x=(x_1,x_2),\; p=(p_1, p_2),\ \psi\in {\mathcal H}_{\rm osc}^1.
 $$
 For $r\geq 0$, $ {\mathcal H}_{\rm osc}^r$ is the weighted Sobolev space associated with the harmonic oscillator $H_0:= \hat p^2 + \hat x^2$,
 $$
 {\mathcal H}_{\rm osc}^r = \{\psi\in L^2(\R^2): H_0^{r/2}\psi\in L^2(\R^2)\},
 $$
endowed  with the norm $\Vert\psi\Vert_r = \Vert H_0^{r/2}\psi\Vert_{ L^2(\R^2)}$.\\
 Recall that we are working here with polynomials classical Hamiltonians of degree  at most 2 so the correspondence classical-quantum is exact. That means
 $$
 \begin{pmatrix} \hat x(t)\\\hat p(t)\end{pmatrix} =
 \Phi_B(t) \begin{pmatrix}\hat x\\\hat p\end{pmatrix}\ ,
 $$
{where $(\hat x(t),\hat p(t))$ is the solution of the Heisenberg equation.}

 Our first {quantum} corollary regards the existence of solutions of the quantum Landau Hamiltonian undergoing unbounded growth of Sobolev norms. 
 {Computing  $ \Phi_B(t)$ and using Corollary 1.3 we get }
\begin{corollary} Let $\omega\in{\mathcal C}_\eps$.
We have, in the operator sense, 
$$
\hat x_1(t) =\hat x_1 +\alpha c_\omega\eps^2t\,\hat p_1 + (1+\eps^4t)\left(A_{\varepsilon,\omega}(t)\cdot \hat x + B_{\varepsilon,\omega}(t)\cdot\hat p\right),
$$
where $\alpha\neq 0$ and $c_\omega$ is given by \eqref{c}, and $\vert A_{\varepsilon,\omega}(t)\vert + \vert B_{\varepsilon,\omega}(t)\vert = O(1)$.
In particular, if $c_\omega\neq0$,  there exists  $K\geq 0$ such that 
 for  any $\psi\in {\mathcal H}_{\rm osc}^1$  we have 
 \bea
\Vert \hat x_1(t)\psi\Vert_0  \geq \alpha c_\omega t\eps^2\Vert D_{x_1}\psi\Vert_0 - K(1 +\eps^4t)\Vert\psi\Vert_1.
\eea
 In particular, if $\eps$ is sufficiently small,  then
$\Vert\hat x_1(t)\psi\Vert_0\nearrow +\infty$ as $t\nearrow +\infty$.\\
We have also a lower bound for the quantum average of the time evolution of the  position observable, for $\psi\in {\mathcal H}_{\rm osc}^{1/2}$ :
\bea
\vert\langle\psi, \hat x_1(t)\psi\rangle\vert \geq \alpha c_\omega t\eps^2\vert \langle\psi, D_{x_1}\psi\rangle\vert - K(1 +\eps^4t)\Vert\psi\Vert_{1/2}.
\eea
\end{corollary}

%
%
%

Our second corollary regards the  Floquet  spectrum of the time quasi-periodic Hamiltonian $H_L(t)$.

\begin{corollary}The quantum dynamics ${\mathcal U}_{\eps,\omega}(t,0)$ of $H_L(t)$  is conjugated to the quantum dynamics ${\rm e}^{-\im t \tilde H_{L,\eps,\infty}}$ of  the stationary Hamiltonian
 $\tilde H_{L,\eps,\infty}:= b(\eps)(D_{\tilde x_1}^2 +\tilde x_1^2) + c(\eps)D_{\tilde x_2}^2$.\\
 Moreover  as far as $b(\eps)>0$ and $c(\eps)>0$,  the spectrum
 $\sigma(H_{L,\eps,\infty})$ of  $\tilde H_{L,\eps,\infty}$  is absolutely continuous, with thresholds at  the
 Landau levels, 
 $$
 \sigma(H_{L,\eps,\infty}) = \bigcup_{j\geq 0}[b(\eps)(j+1/2),
   +\infty[ .
     $$
\end{corollary}
\begin{proof}  Denote by  $h_{L,\eps,\infty}:= b(\eps)|z_1|^2+c(\eps)(z_2-\bar z_2)^2$ the stationary classical  Hamiltonian  to which $h_L(t)$ is conjugated, see 
\eqref{bc.1}. In the real coordinates $(x,p)$ 
 we have $h_{L,\eps,\infty}(x,p) = \frac{b(\eps)}{2}(p_2^2 + x_2)^2 +  \frac{c(\eps)}{2}(x_1-p_2)^2$. 
 By the symplectic change of coordinates
 $\tilde x_1=x_1-p_2, \tilde p_1 =p_1$, $\tilde x_2=x_2 -p_1, \tilde p_2=p_2$, we have 
 $$\tilde h_{L,\eps,\infty}(\tilde x,\tilde p)= b(\eps)(\tilde p_2^2+ \tilde x_2)^2+c(\eps) \tilde p_1^2. $$The first part of the Corollary follows.\\
 For the second part notice that we have the following family of generalized eigenfunctions:
 $ \Psi_{j,\xi}(x_1,x_2) = \psi_j(x_1){\rm e}^{\im x_2\xi}$ such that we have
 $$
 \tilde H_{L,\eps,\infty} \Psi_{j,\xi} = (b(\eps)(j+1/2) + c(\eps)\xi^2)\Psi_{j,\xi}.
 $$
Hence we get a description  of the spectrum of $\tilde H_{L,\eps,\infty}$. 
\end{proof}

\subsection{Main result in symmetric gauge}\label{symgauge}

We still consider a magnetic Schr\"odinger operator with a time-quasiperiodic  magnetic field, i.e.  $B(t)=B_0+\eps f(\omega t)$ with  $f$ real analytic on the torus ${\mathbb T}^n$, and $\hat f(0)=0$,  but now in the symmetric gauge, namely
\begin{equation}
\label{landauSym}
H_{sL}(t) = \left(D_{x_1} - \frac12B(t)x_2\right)^2   + \left(D_{x_2} +\frac12B(t)x_1\right)^2  \ , 
\end{equation}
and the frequency vector $\omega$ in the set of nonresonant frequencies $\D_0$ of Definition
\ref{diof}.\\
We can write
$$H_{sL}(t)= H_{sL}+ R_{sL}(\omega t)$$
where 
\begin{align*}
H_{sL}&= \left(D_{x_1} - \frac12B_0x_2\right)^2   + \left(D_{x_2} +\frac12B_0x_1\right)^2,\\
R_{sL}(\omega t)&= \eps f(\omega t) \left(x_1 (D_{x_2}+\frac{B_0}2 x_1)  -x_2 (D_{x_1}-\frac{B_0}2 x_2)\right) +\frac14\eps^2f(\omega t)^2 (x_1^2+x_2^2).\end{align*}
We denote by $h_{sL}(t)$ the corresponding classical Hamiltonian:
\begin{equation}
\label{landausym}h_{sL}(t)= \big(p_1-\frac{B(t)}2x_2\big)^2+\big(p_2+\frac{B(t)}2x_1\big)^2=h_{sL}+r_{sL}(\omega t).\end{equation}
We introduce the symplectic variables
$$p'_1=x_1\ ,\ x'_2=x_2\ \ x_1'=-p_1\ ,\ p_2'=p_2 
$$
and in these variables $h_{sL}$ reads
$$h_{sL}=\left(x'_1+\frac{B_0}2x'_2\right)^2+\left(p'_2+\frac{B_0}2p'_1\right)^2.$$
Then in the new symplectic variables $(y_1,y_2,\eta_1,\eta_2)$ defined by
$$y_1=\frac1{\sqrt{B_0}}x'_1+ \frac{\sqrt{B_0}}2 x'_2,\quad \eta_1=\frac{\sqrt{B_0}}2p'_1+\frac1{\sqrt{B_0}}p'_2$$
and
$$y_2=\frac1{\sqrt{B_0}}x'_1- \frac{\sqrt{B_0}}2 x'_2,\quad \eta_2=\frac{\sqrt{B_0}}2p'_1-\frac1{\sqrt{B_0}}p'_2$$
we obtain that $h_{sL}$ is the {\em degenerate} two-dimensional Harmonic oscillator
$$h_{sL}=B_0(y_1^2+\eta_1^2)=2B_0 |z_1|^2$$
where $z_1=\frac{y_1+\im \eta_1}{\sqrt2}$ and similarly $z_2=\frac{y_2+\im \eta_2}{\sqrt2}$. We denote by $\tau_0$ the linear symplectic   transformation from $\R^4$ to $\C^2$ defined by
 $(z_1,z_2) = \tau_0(x,p)$.
We note that in the complex variables the symplectic form reads: $$dy_1\wedge d\eta_1+dy_2\wedge d\eta_2=\im (dz_1\wedge d\bar z_1+dz_2\wedge d\bar z_2).$$
In order to state the next result we introduce 
\be\label{d}
d_\omega=\int_{\T^n} h(\theta)^2 d\theta
\ee
where
\begin{equation}
  \label{d.1}
h(\theta)=-\frac1{\sqrt{2B_0}}\sum_{k\neq 0}\frac{\omega\cdot k+2\im
  B_0}{\omega\cdot k+2 B_0}\hat f(k) e^{\im k\cdot \theta}.
\end{equation}
\begin{remark}
  \label{dnonZero}
  The number $d_\omega$ is well defined and real for  $\omega \in \D_0$.
\end{remark}
Our next result shows that,  in the symmetric gauge, the perturbed classical Hamiltonian
$h_{sL}(t)$ in \eqref{landausym} is conjugated to a two dimensional Harmonic oscillator which is {\em non-degenerate} provided the constant $d_\omega \neq 0$.
\begin{theorem} 
  \label{len.mag.sym} There exists $\eps_0>0$ such that for $|\eps|<\eps_0$, 
   there exist
   \begin{itemize}
   \item an asymptotically full measure set of frequencies $\mathcal C_\eps\in \D_0$ satisfying:
   $$\lim_{\eps\to 0}{\text{meas}(\D_0\setminus\mathcal C_\eps)}=0;$$
   \item a linear, symplectic change of variable $\tau$,   depending on $\omega, t,\eps$,  which is close to the identity, namely  $\tau = \mathbb{I} + O(\eps)$
   uniform in the other parameters,
   \end{itemize}
    such that, for $\omega\in\mathcal C_\eps$, and provided  $d_\omega$ in  \eqref{d} does not vanish, 
\begin{equation}\label{hsL.n}
{(h_{sL}}(t)\circ \tau_0^{-1})\circ \tau =b(\eps)|z_1|^2+d(\eps)|z_2|^2
\end{equation}
with
$$b(\eps)= 2B_0+O(\eps^2),\quad d(\eps)= d_\omega\eps^2+O(\eps^4)$$
and  $d_\omega$ is given by \eqref{d} and does not vanish  for $\omega\in\mathcal C_\eps$.
\end{theorem}

In the new coordinates the motion of \eqref{hsL.n} is easily computed:
\bea
z_1(t) &=& {\rm e}^{-\im b(\eps)t}z_1(0)\nonumber \ , \qquad 
z_2(t) = {\rm e}^{-\im d(\eps)t}z_2(0).\nonumber
\eea
In particular, in the symmetric gauge, provided the constant  $d_\omega$ in \eqref{d} does not vanish, all the trajectories are bounded, contrary to what happens in the Landau gauge.

Also in this case we are able to describe the quantum flow, which in this case  is uniformly bounded in any Sobolev space  ${\mathcal H}_{\rm osc}^r $.\\ 
Let  ${\mathcal U}_{\eps, \omega}(t,s)$ be the quantum  propagator defined by the Hamiltonian $H_{sL}(t)$  in \eqref{landauSym}. 
So we have $\im\partial_t{\mathcal U}_{\eps, \omega}(t,s)= H_{sL}(t){\mathcal U}_{\eps, \omega}(t,s)$, ${\mathcal U}_{\eps, \omega}(s,s)= {\mathbb I}$.
\begin{corollary} 
\label{cor:sym}
 There exists $\eps_0>0$ such that for $\vert\eps\vert \leq \eps_0$,  for any $r > 0$  there exist $0<c_r\leq C_r$ such that if $\om\in \mathcal C_\eps$  we have for any $\psi_0\in  {\mathcal H}_{\rm osc}^r $, 
\beq\label{qpropag}
    c_r\Vert\psi_0\Vert_r \leq \Vert{\mathcal U}_{\eps,\omega}(t,0)\psi_0\Vert_r \leq  C_r\Vert\psi_0\Vert_r,\; \forall t\in\R.
\eeq
\end{corollary}
\begin{proof} We follow the proof given in \cite{BGMR1}, Corollary 1.3.  We shall give here only the main steps.\\
  For simpler notation we assume that $B_0=1$. Let 
  $$
  H_{sL,\eps,\infty} = \frac{b(\eps)}{2}((\hat p_1-\hat x_2)^2 +\hat p_2^2) +  \frac{d(\eps)}{2}((\hat p_2-\hat x_1)^2 +\hat p_1^2),
  $$
  and $Z_1=(\hat p_1-\hat x_2)^2 +\hat p_2^2$, $Z_2=(\hat p_2-\hat x_1)^2 +\hat p_1^2 $.  Notice that $Z_j = \hat z_j^*\hat z_j$. Moreover
  $[Z_1,  Z_2]=0$. So  $[H_{sL,\eps,\infty}, Z_1+Z_2]=0$. But  $Z_1+Z_2$ is a non degenerate harmonic oscillator hence 
  ${\rm e}^{-\im tH_{sL,\eps,\infty}}\psi_0$ satisfies the estimates \eqref{qpropag}. Then as in \cite{BGMR1}, from the classical KAM construction  there exists
  $$
 \chi_{\eps,\omega}(t,x,p)=\begin{pmatrix} x \\  p \end{pmatrix}\cdot S_{\eps,\omega}(t)\begin{pmatrix} x \\ p \end{pmatrix},
  $$
 where $S_{\eps,\omega}(t)$ is a symmetric matrix with uniformly bounded entries. Let $U_{\eps,\omega}(t) = {\rm e}^{\im \eps\hat\chi_{\eps,\omega}(t)}$ we have
 $$
 {\mathcal U}_{\eps, \omega}(t,0) = U^*_{\eps,\omega}(t){\rm e}^{-\im tH_{\eps,\infty}}U_{\eps,\omega}(t).
 $$
Using that (\cite{BGMR1}, Theorem 2.7), uniformly in $(t,\eps,\omega)$, we have
$$
\tilde c_r\Vert\psi\Vert_r \leq \Vert U^*_{\eps,\omega}(t)\psi\Vert_r\leq \tilde C_r\Vert\psi\Vert_r, 
$$
we get \eqref{qpropag}.
\end{proof}

\begin{corollary}
 The  symmetric Landau Hamiltonian $H_{sL}(t)$ is reducible to a stationary Hamiltonian  $H_{sL,\eps,\infty}$ with a discrete spectrum with all eigenvalues of finite multiplicities as far $b(\eps)>0, d(\eps)\neq 0$. Moreover, up to a linear symplectic transformation, $H_{sL,\eps,\infty}$ is combination of  two 1-d harmonic oscillators
 $$
 H_{sL,\eps,\infty}= b(\eps)(D_{x_1}^2+x_1^2) + d(\eps)(D_{x_2}^2+x_2^2).
 $$
 \end{corollary}
 \proof  Keeping the notations of the proof of Corollary \ref{cor:sym}, 
 since $Z_1+Z_2$ has compact inverse, it has pure point
    spectrum and there exists a basis of eigenfunctions, which, since
    $[Z_1,Z_2]=0$ can be choose in such a way that they are also
    eigenfunctions of both $Z_1$ and $Z_2$. Thus, if one denotes by
    $\lambda_{i,j} = j+1/2$, $j \in\N$ the eigenvalues of $Z_i$, one
    has that $H_{sL,\eps,\infty}$ is diagonal
in the same basis,  with eigenvalues $\mu_{j_1,j_2}(\eps)=
b(\eps)(j_1+1/2)+ d(\eps)(j_2+1/2)$. So the  symmetric Landau
Hamiltonian $H_{sL}(t)$ is reducible to a stationary Hamiltonian
$H_{sL,\eps,\infty}$ with a discrete spectrum with all eigenvalues of
finite multiplicities as far $b(\eps)>0, d(\eps)\neq 0$.  \endproof

\subsection{About the change of gauge}\label{changegauge}

The fact that one has a completely different behaviour in the
  case of the Landau gauge and in the case of the symmetric gauge
  could seem surprising at first sight, however we remark that they
  correspond to different physical situations. Indeed, the electric
  and the magnetic field are given by 
\begin{equation}\label{gauge}{\bold B}=\nabla_x \times A(x,t)\quad
  \text{ and }\quad {\bold E}=-\frac{\partial A}{\partial
    t}(x,t)-\nabla_x V(x,t).
\end{equation}
Thus, in the time dependent case,  they differ  for the case of
the Landau gauge and the case of the symmetric gauge. As it is well known,
there exists a gauge transformation which allows to pass from one
gauge to the other by keeping the same electromagnetic
fields. For example the electric and magnetic fields can be chosen as 
$$
V(x,t)=0\ ,\quad A(x,t)=B(t)(x_2,0,0)
$$
or as 
$$
V(x,t)=-\frac{B'(t)}2x_1x_2\ ,\quad A(x,t):=\frac{B(t)}2(x_2,-x_1,0) \ . 
$$
The first corresponds to the Landau gauge, while the second
corresponds to the symmetric gauge plus a scalar potential.

\smallskip

\subsection{Related literature}  

Most of the literature about the Landau Hamiltonian regards the
asymptotic behavior of the perturbed  spectrum  under 
 time {\em independent} perturbations, for example scalar potentials in different  classes.
These works ensure conditions on the perturbation so that the perturbed spectrum is asymptotically localized around the Landau levels $\{ 2B_0(j+\frac12)\}_{j \in \N}$, a fact which is not trivial due to the infinite multiplicity of these.
 We mention for example the works \cite{Tag, Raikov, Push,  Lun, Lun2, Her} and reference therein.

The case of time {\em dependent} perturbations, such as \eqref{LS},  is less studied. We mention \cite{Y,  Y2, AnFa} which prove the existence of the  quantum flow, and \cite{MaRo, BGMR2, LZZ2} giving time upper bounds on the dynamics.
The present paper aims to  prove finer properties about the quantum dynamics, and in particular investigates the dicotomy of  ``existence of solutions with unbounded trajectories'' vs ``all trajectories are  bounded''. 
This question has received, in the last decade, a lot of attention. 

In  case of linear, time dependent Schr\"odinger equations,  such as \eqref{LS}, 
the first result about existence of solutions with unbounded paths is due to Bourgain  \cite{bourgain99} on the torus. 
Recently,  several works have considered non-degenerate Harmonic oscillators on $\R^d$ and constructed time dependent perturbations in the form of  pseudodifferential operators  \cite{del,  Mas19}, polynomial functions \cite{BGMR1,  LLZ1, LLZ,  LLZ2}, or classical potentials   \cite{ FaouRaph, Thomann},  that create   solutions with unbounded trajectories. We also cite the recent results \cite{Mas21, Mas22} which prove that generic,  time periodic,  pseudodifferential perturbations provoke instability phenomena.

 On the opposite side, many  works prove that, 
when the perturbation is  small in size and quasi-periodic in time with a non resonant frequency $\omega$, 
 all trajectories are bounded in Sobolev spaces.  There results are based on 
   KAM reducibility methods ensuring that the  linear propagator has operatorial norm (in Sobolev spaces)  bounded  uniformly in time, in the same spirit of our Corollary \ref{cor:sym}.
This is the case,  in  great generality,  for systems in 1-spatial dimensions: limiting ourselves to results considering perturbations of the Harmonic oscillator on $\R$, we cite 
   \cite{C87, Wan08, GT11,  Bam18, Bam17}.
In a  higher dimensional setting,    such as the one of equation \eqref{LS},  there are few  KAM reducibility results.  
We cite \cite{EK,PP} for the Schr\"odinger on $\T^d$, 
\cite{GP16, BGMR1} for the Harmonic oscillators on $\R^d$, 
\cite{Mon19} for the wave on $\T^d$, and
\cite{FGMP, BLM} for transport equations on $\T^d$.

%
%
%

\section{Proofs of main Theorems}
\subsection{Proof of Theorem \ref{len.mag}} \label{proofLand}
We prefer to work in the extended phase space in which we add the angles
$\theta\in\T^n$ as new variables and their conjugated momenta
$I\in\R^n$. So our phase space is now $\T^n\times \R^n\times \C^4\ni (\theta,I,z_1,z_2)$, with
$\C^2$ considered as a real vector space. The symplectic form is $dI\wedge d\theta+\im dz\wedge d\bar z
$ and the Hamiltonian equations of a Hamiltonian function
$h(\theta,I,z_1,z_2)$ are
\begin{equation}\label{Ham}
\dot I=-\frac{\partial h}{\partial \theta}\ ,\quad \dot
\theta=\frac{\partial h}{\partial I}\ ,\quad\dot
z_1=-\im\frac{\partial h}{\partial \bar z_1}\  ,\quad\dot z_2=-\im\frac{\partial h}{\partial \bar z_2}\ .
\end{equation}
In this framework the Hamiltonian equation associated with the classical time dependent Hamiltonian function  $h_L$ in \eqref{landaub} is equivalent to the autonomous Hamiltonian system \eqref{Ham} with   $$h=h_0+r_1+r_2$$ and
\begin{align}
\label{h00} h_0&=\omega \cdot I +2B_0|z_1|^2,\\
\label{r1}r_1&= \eps (z_1+\bar z_1)(z_1+\bar z_1-\im(z_2-\bar z_2))f(\theta),\\
\label{r2}r_2&=\frac{\eps^2}{2B_0}(z_1+\bar z_1-\im(z_2-\bar z_2))^2f(\theta)^2.
\end{align}
The proof of Theorem \ref{len.mag} follows a KAM strategy: we want to eliminate the angles in $h$ by canonical changes of variables. This canonical changes of variables will be constructed as time 1 flows, $\Phi^1_\chi$, of some Hamiltonian $\chi$. We begin by computing explicitly the first two KAM steps and then we will be in position to apply a KAM theorem with symmetry, namely Theorem \ref{KAMclassico}. First we construct the first change of variables and we begin by solving a so called homological equation.
\begin{lemma}\label{lem:1}
Let 
\begin{align*}\chi_1:=&\im \eps\sum_{k\in\Z^n\setminus \{0\}}\hat f(k) e^{\im k\cdot \theta}\left(\frac{z_1^2}{\omega\cdot k+4B_0}+\frac{\bar z_1^2}{\omega\cdot k-4B_0}+2\frac{z_1\bar z_1}{\omega\cdot k}\right)\\ 
+ &\eps (z_2-\bar z_2)\sum_{k\in\Z^n\setminus \{0\}}\hat f(k) e^{\im k\cdot \theta}\left(\frac{z_1}{\omega\cdot k+2B_0}+\frac{\bar z_1}{\omega\cdot k-2B_0}\right) \ , 
\end{align*}
then $\chi_1$ solves the following homological equation:
\begin{equation}\label{homo1} \{\chi_1, h_0\}+r_1=0.\end{equation}
\end{lemma}
\proof

First  recall that
\begin{equation}\label{poisson}
  \{F,G\}:=\sum_{j=1}^n \frac{\partial  F}{\partial \theta_j}\frac{\partial  G}{\partial I_j}-\frac{\partial  F}{\partial I_j}\frac{\partial  G}{\partial \theta_j}+\im \sum_{j=1,2} \frac{\partial  F}{\partial z_j}\frac{\partial  G}{\partial \bar z_j}-\frac{\partial  G}{\partial z_j}\frac{\partial  F}{\partial \bar z_j} .
\end{equation}
Then we introduce 
$$N(\theta, z_1)=\sum_{k\in\Z^n\setminus \{0\}}\hat f(k) e^{\im k\cdot \theta}\left(\frac{z_1}{\omega\cdot k+2B_0}+\frac{\bar z_1}{\omega\cdot k-2B_0}\right)$$
and
$$M(\theta, z_1)=\sum_{k\in\Z^n\setminus \{0\}}\hat f(k) e^{\im k\cdot \theta}\left(\frac{z_1^2}{\omega\cdot k+4B_0}+\frac{\bar z_1^2}{\omega\cdot k-4B_0}+2\frac{z_1\bar z_1}{\omega\cdot k}\right)$$
in such way we have
$$\chi_1= \im\eps M+\eps(z_2-\bar z_2) N.$$
So, since $h_0$ in \eqref{h00} doesn't depend on $z_2$, we get
$$ \{\chi_1, h_0\}=\im\eps\{M,h_0\}+\eps (z_2-\bar z_2)\{N,h_0\}.$$
Then we compute
\begin{align*}
\{N,h_0\}&=\im \sum_{k\in\Z^n\setminus \{0\}} (k\cdot \omega) \hat f(k) e^{\im k\cdot \theta}\left(\frac{z_1}{\omega\cdot k+2B_0}+\frac{\bar z_1}{\omega\cdot k-2B_0}\right)\\
&+\im\sum_{k\in\Z^n\setminus \{0\}} \hat f(k) e^{\im k\cdot \theta}\left(\frac{2B_0z_1}{\omega\cdot k+2B_0}+\frac{-2B_0\bar z_1}{\omega\cdot k-2B_0}\right)\\
&=\im (z_1+\bar z_1) f(\theta),
\end{align*}
and
\begin{align*}
\{M,h_0\}&= \im \sum_{k\in\Z^n\setminus \{0\}} (k\cdot \omega) \hat f(k) e^{\im k\cdot \theta}\left(\frac{z_1^2}{\omega\cdot k+4B_0}+\frac{\bar z_1^2}{\omega\cdot k-4B_0}+2\frac{z_1\bar z_1}{\omega\cdot k}\right)\\
&+\im \sum_{k\in\Z^n\setminus \{0\}} \hat f(k) e^{\im k\cdot \theta}\left(\frac{4B_0z_1^2}{\omega\cdot k+4B_0}+\frac{-4B_0\bar z_1^2}{\omega\cdot k-4B_0}\right)\\
&=\im (z_1+\bar z_1)^2 f(\theta) \ , 
\end{align*}
from which  \eqref{homo1} follows.
\endproof
 Then since,
 $$\{\chi_1,h_0\}+r_1=0,\quad \{\chi_1,\{\chi_1,h_0\}\}=-\{\chi_1,r_1\},$$
 we get
 \begin{equation}\label{hrondphi}h\circ \Phi^1_{\chi_1}=h+\{\chi_1,h\}+\frac12 \{\chi_1,\{\chi_1,h_0\}\}+O({\eps^3})=  h_0+\frac12 \{\chi_1,r_1\}+r_2+O({\eps^3}).
 \end{equation}
 Next we wish to compute explicitly $\{\chi_1,r_1\}$. In view of the expressions of $\chi_1$ and $r_1$, we first compute 
\begin{align*}\{M,(z_1+\bar z_1)^2\}&=2(z_1+\bar z_1)\{M,z_1+\bar z_1\},\\
\{M,(z_1+\bar z_1)(z_2-\bar z_2)\}&=(z_2-\bar z_2)\{M,z_1+\bar z_1\},\\
\{M,z_1+\bar z_1\}&=2\im\sum_{k\in\Z^n\setminus \{0\}}\hat f(k)e^{ik\cdot \theta}\left( \frac{z_1}{\omega\cdot k+4B_0}-\frac{\bar z_1}{\omega\cdot k-4B_0}+\frac{\bar z_1-z_1}{\omega\cdot k}\right)\\
&=-8\im B_0\sum_{k\in\Z^n\setminus \{0\}}\hat f(k)e^{ik\cdot \theta}\left( \frac{z_1}{\omega\cdot k(\omega\cdot k+4B_0)}+\frac{\bar z_1}{\omega\cdot k(\omega\cdot k-4B_0)}\right)
\end{align*}
 and
\begin{align*}
\{N,(z_1+\bar z_1)^2\}&=2(z_1+\bar z_1)\{N,z_1+\bar z_1\},\\
\{N,(z_1+\bar z_1)(z_2-\bar z_2)\}&=(z_2-\bar z_2)\{N,z_1+\bar z_1\}, \\
\{N,z_1+\bar z_1\}&=-4\im B_0 \sum_{k\in\Z^n\setminus \{0\}} \frac{\hat f(k)e^{\im k\cdot \theta}}{(\omega\cdot k)^2-4B_0^2}.
\end{align*}
Therefore

\begin{align}
\notag
\{\chi_1,r_1\}&=\{\im\eps M+\eps(z_2-\bar z_2) N,\eps (z_1+\bar z_1)(z_1+\bar z_1-\im(z_2-\bar z_2))f(\theta)\}\\
\notag
&=2\im \eps^2f(\theta)(z_1+\bar z_1)\{M,(z_1+\bar z_1)\} +\eps^2(z_2-\bar z_2)f(\theta)\{M,(z_1+\bar z_1)\}\\
\notag
&+2\eps^2f(\theta)(z_2-\bar z_2)(z_1+\bar z_1)
\{N,z_1+\bar z_1\} -\im\eps^2 (z_2-\bar z_2)^2f(\theta)\{N,z_1+\bar z_1\}  \\
\notag
&=\eps^2f(\theta)(2\im(z_1+\bar z_1)+(z_2-\bar z_2))\{M,(z_1+\bar z_1)\}\\
\notag
&-\im \eps^2f(\theta)(z_2-\bar z_2)(2\im(z_1+\bar z_1)+ (z_2-\bar z_2)) \{N,z_1+\bar z_1\}\\
\notag
&=\eps^2f(\theta)(2\im(z_1+\bar z_1)+(z_2-\bar z_2))(\{M,(z_1+\bar z_1)\}-\im(z_2-\bar z_2)\{N,z_1+\bar z_1\})\\
\notag
&=-4\im\eps^2B_0f(\theta)(2\im(z_1+\bar z_1)+(z_2-\bar z_2)) \times\\
\label{chi1r1}
&\times \sum_{k\in\Z^n\setminus \{0\}}\hat f(k)e^{\im k\cdot \theta}\left( \frac{2z_1}{\omega\cdot k(\omega\cdot k+4B_0)}+\frac{2\bar z_1}{\omega\cdot k(\omega\cdot k-4B_0)}-\im \frac{z_2-\bar z_2}{(\omega\cdot k)^2-4B_0^2}  \right).
\end{align}

At the next KAM step we remove from $\frac12 \{ \chi_1, r_1\} + r_2$ all the $\eps^2$ terms except resonant monomials, i.e. in our case, $|z_1|^2$ and $(z_2-\bar z_2)^2$.
In view of the expressions \eqref{chi1r1}, \eqref{r2}, we thus obtain 
 \begin{equation}\label{h2}h_2=h\circ \Phi^1_{\chi_1}\circ \Phi^1_{\chi_2}= h_0+a_\omega\eps^2|z_1|^2+c_\omega\eps^2(z_2-\bar z_2)^2+O({\eps^3})\end{equation}
 where 
 \begin{align*}c_\omega&= {-} \sum_{k\in\Z^n\setminus \{0\}}\hat f(k)\hat f(-k)\left(\frac{2B_0}{(\omega\cdot k)^2-4B_0^2}  \right)-\frac1{2B_0}\langle f^2\rangle\\
 &={-} \sum_{k\in\Z^n\setminus \{0\}}\hat f(k)\hat f(-k)\left(\frac{2B_0}{(\omega\cdot k)^2-4B_0^2}{+} \frac1{2B_0}  \right)\\
 &= -\frac1{2B_0}\sum_{k\in\Z^n\setminus \{0\}}\hat f(k)\hat f(-k) \frac{{(\omega\cdot k)^2}}{(\omega\cdot k)^2-4B_0^2}  
 \end{align*}
 as stated in \eqref{c} while
  \begin{align*}a_\omega&=8 B_0\sum_{k\in\Z^n\setminus \{0\}}\hat f(k)\hat f(-k)\left(\frac{1}{(\omega\cdot k)((\omega\cdot k)+4B_0)}+\frac{1}{(\omega\cdot k)((\omega\cdot k)-4B_0)}  \right)+\frac1{B_0}\langle f^2\rangle\\
 &=\sum_{k\in\Z^n\setminus \{0\}}\hat f(k)\hat f(-k)\left(\frac{16 B_0}{(\omega\cdot k)^2-16B_0^2}+\frac1{B_0}  \right)\\
 &= \frac 1{B_0}\sum_{k\in\Z^n\setminus \{0\}}\hat f(k)\hat f(-k) \frac{(\omega\cdot k)^2}{(\omega\cdot k)^2-16B_0^2}  .
 \end{align*}

 To end the proof of Theorem \ref{len.mag} we just have to iterate this KAM step and to prove the convergence of such a process. In particular we will check  that we only remain with resonant monomials, i.e. in our case, $|z_1|^2$ and $(z_2-\bar z_2)^2$.\\
 {Concretely Theorem \ref{len.mag} is obtained by applying Theorem
  \ref{KAMclassico} to the Hamiltonian \eqref{h2} taking
  $\nu_1:=2B_0+a_\omega \eps^2$, $c_0:=c_\omega \eps^2$ and
  $\eps_0:=\eps^3$.} 
 We stress out that the difference between the estimate \eqref{nuc0} in Theorem \ref{KAMclassico} and the estimate \eqref{bc} in Theorem \ref{len.mag}  is due to the specific form  of $r$ in \eqref{r1} and the fact that $\hat f(0)=0$.
 
 \subsection{Proof of Theorem \ref{len.mag.sym}} 
 We still work in the same framework as in section \ref{proofLand}, i.e. in the extended phase space in which we add the angles
$\theta\in\T^n$ as new variables and their conjugated momenta
$I\in\R^n$. So our phase space is still $\T^n\times \R^n\times \C^2\ni (\theta,I,z_1,z_2)$ (see \eqref{Ham}). 

In this framework the Hamiltonian equation associated with the classical time dependent Hamiltonian function \eqref{landausym} $h_{sL}(t)$ is equivalent to the autonomous Hamiltonian system \eqref{Ham} with   $h=h_0+r$ where
$$h_0=\omega \cdot I +2B_0|z_1|^2$$
and $r=r_{sL}(\theta)$. To compute explicitly this term, recall that, in the coordinates introduced in section \ref{symgauge}, 
\begin{align*}x_2&=x'_2=\frac1{\sqrt{B_0}}(y_1-y_2)=\frac1{\sqrt{2B_0}}(z_1+\bar z_1-(z_2+\bar z_2)),\\
x_1&=p'_1=\frac1{\sqrt{B_0}}(\eta_1+\eta_2)=\frac1{\im\sqrt{2B_0}}(z_1-\bar z_1+(z_2-\bar z_2)),\\
p_1-\frac{B_0}2x_2&=-x'_1-\frac{B_0}2 x'_2=-\sqrt{B_0}y_1=-\frac{\sqrt{B_0}}{\sqrt2}(z_1+\bar z_1),\\
p_2+\frac{B_0}2x_1&=p_2-\frac{B_0}2 \xi_1=\sqrt{B_0}\eta_1=\frac{\sqrt{B_0}}{\im\sqrt2}(z_1-\bar z_1). 
\end{align*}
Therefore $r(\theta)$ reads 
\begin{align*}r(\theta)&=\eps f(\theta)(y_1(y_1-y_2)+\eta_1(\eta_1+\eta_2))+\frac{\eps^2}{4B_0} f(\theta)^2((y_1-y_2)^2+(\eta_1+\eta_2)^2)\\
&=\frac{\eps}{2}f(\theta)((z_1+\bar z_1)(z_1+\bar z_1-(z_2+\bar z_2))-(z_1-\bar z_1)(z_1-\bar z_1+(z_2-\bar z_2)))\\
&+\frac{\eps^2}{8B_0} f(\theta)^2( (z_1+\bar z_1-(z_2+\bar z_2))^2-  (z_1-\bar z_1+(z_2-\bar z_2))^2)\\
&=\eps r_1+\eps^2 r_2.
\end{align*}
We follow the same strategy than in the previous section and we are interested by the quadratic terms in $z_2$, $\bar z_2$ after the second KAM step. At the first step (at order $\eps$) we don't have such term (in $r_1$) and thus we eliminate all the  terms of order $\eps$ by a symplectic change of variables $\Phi_{\chi_1}^1$ where $\chi_1$ is the solution of the homological equation.
$$
\{\chi_1,h_{sL}\}=-r_1 \ . 
$$
Since 
$$r_1=-\eps f(\theta)(z_1z_2+\bar z_1\bar z_2)+\text{ quadratic terms in }z_1,\ \bar z_1$$
we take (with $\nu=2B_0$)
\begin{align*}\chi_1&=-\im\eps\sum_{k\neq0}\hat f(k) e^{ik\cdot \theta}\left( \frac{z_2 z_1}{k\cdot\omega +\nu} +\frac{\bar z_2 \bar z_1}{k\cdot\omega -\nu} \right)+\text{ quadratic terms in }z_1,\ \bar z_1.
\end{align*}
Then, using \eqref{hrondphi}, the quadratic terms in $z_2$, $\bar z_2$ after the second KAM step come from the quadratic terms in $z_2$, $\bar z_2$ in $\frac12\{\chi_1,r_1\}$ and in $r_2$.\\
From $r_2$ we get 
$$\frac{\eps^2 f(\theta)^2}{4B_0}z_2\bar z_2$$
and from $\frac12\{\chi_1,r_1\}$ we get
$$-\frac12\eps^2f(\theta) \sum_{k\neq0}\hat f(k)e^{\im k\cdot\theta}\left( \frac{\bar z_2 z_2}{k\cdot\omega +\nu} -\frac{ z_2 \bar z_2}{k\cdot\omega -\nu} \right).$$
Now the second KAM step will eliminate all the  $e^{\im k\cdot\theta}z_2\bar z_2$ for $k\neq0$ and therefore we obtain like in \eqref{h2}
\begin{equation}\label{hs2}h_2=h\circ \Phi^1_{\chi_1}\circ \Phi^1_{\chi_2}= \omega \cdot I +(2B_0+a\eps^2)|z_1|^2+d_\omega\eps^2|z_2|^2+O(\eps^3)\end{equation}
where
\begin{align*}d_\omega&=\frac1{4B_0}\sum_{\ell}\hat f(\ell)\hat f(-\ell)+\sum_{\ell\neq0} \hat f(\ell)\hat f(-\ell)\frac{2B_0}{(\ell\cdot\omega)^2-4B_0^2}\\&=\frac1{4B_0}\sum_{\ell\neq0} \hat f(\ell)\hat f(-\ell)\frac{(\ell\cdot\omega)^2+4B_0^2}{(\ell\cdot\omega)^2-4B_0^2}\end{align*}
as stated in \eqref{d}.\\
Then, if $d_\omega\neq0$, $h_2$ appears as a perturbation of the non degenerate Hamiltonian $(2B_0+a\eps^2)|z_1|^2+d_\omega\eps^2|z_2|^2$ and we can apply Theorem \ref{KAMclassico1}.

 \section{Proofs of two  reducibility Theorems} \label{KAM}
 In this section we prove the two reducibility theorems that we need
 to conclude the proofs of Theorem \ref{len.mag} and Theorem
 \ref{len.mag.sym}.

First, it is usefull to change the notation in order to make clear
that the variable $\bar z$ is not necessarily the complex conjugate of
$z$: we only have to define the concept of real submanifold of the phase space and it depends on the variables we use (see Definition \ref{reali} and Remark
\ref{reali.2} below for more details).

So, first of all we define
\begin{equation}
    \label{inivar}
\csi_j:=z_j\ ,\quad \eta_j:=\bar z_j\ ,\quad j=1,2\ .
  \end{equation}

\begin{definition}
  \label{reali}
  For $(\csi,\heta)\in C^4$, define the involution
  \begin{equation}
    \label{reali.eq.1}
(\csi,\heta)\mapsto I(\csi,\heta)=(\bar \heta,\bar\csi)\ ,
  \end{equation}
the states $(\csi,\heta)$ s.t.   $(\csi,\heta)=I(\csi,\heta)$ will be
said to be \emph{real}.
\end{definition}
\begin{remark}
  \label{ham.real}
The Hamiltonians we are dealing with are real when $(\csi,\heta)$ is real.
\end{remark}

In order to deal with the Hamiltonian \eqref{h2} whose expansion up to order $\eps^2$  only depends on $z_1$, $\bar z_1$ and $z_2-\bar z_2$, it is useful to introduce the following canonical change of variables
\begin{equation}
   \label{ambio}
\csi'_2:=\csi_2-\heta_2\ ,\quad \heta_2':=\heta_2\ .
\end{equation}
In these variables (and omitting primes) \eqref{h2} reads
 \begin{equation}\label{h2bis}h_2= \omega \cdot I +(2B_0+a_\omega\eps^2)\xi_1\eta_1+c_\omega\eps^2\xi_2^2+O({\eps^3}).\end{equation}
\begin{remark}
  \label{reali.2}
in these variables the involution $I$ takes the form
(omitting the primes)
\begin{equation}
  \label{new.invo}
(\csi_1,\heta_1,\csi_2,\heta_2)\mapsto I(\csi_1,\heta_1,\csi_2,\heta_2
)\equiv (\bar\heta_1,\bar \csi_1, \bar \csi_2,\overline{\csi_2-\heta_2})
  \end{equation}
  and the real submanilfold reads: $\eta_1=\bar \xi_1$, $\xi_2$ is real and $2\Re( \eta_2)=\xi_2$.\\
In all the situations we will encounter the Hamiltonian is
independent of $\heta_2$, therefore the fact that the reality condition 
becomes more complicate will be completely irrelevant.
Of course, in the following a Hamiltonian  expressed in the
variables \eqref{ambio} will be said to be real if it takes  real values  for real 
$(\csi,\heta)$, i.e. for $(\xi,\eta)$ which are fixed points with respect to the involution \eqref{new.invo}.
\end{remark}

The first theorem we will prove  concerns  quasi-periodic in time Hamiltonians of the form
\be \label{ham.cla}h_{\eps_0}(t,\csi, \heta)= \nu_1 \csi_1\heta_1 {+c_0\xi_2^2+\eps_0} q(\om
t,\csi_1,\heta_1,\csi_2)\ee
where  $q$ is a polynomial in $(\csi_1,\heta_1,\csi_2)$ homogeneous of
degree 2 with coefficients that depend quasi periodically on time. The important point is that $q$ does not depend on $\heta_2$. \\
In the following we denote, for $\sigma >0$,  $\T^n_\s {:=\{x+\im y \colon x \in \T^n,
  \, y\in \R^n , \, |y| \leq \s \}}$.

\begin{theorem}
\label{KAMclassico} Assume that $\nu_1>0$  and that $
\T^n\times\C^{4}\ni(\theta,\csi)\mapsto q(\theta,\csi)\in\C$ is a polynomial in
$(\csi_1,\heta_1,\csi_2)$ homogeneous of degree 2, independent of $\eta_2$, with coefficients  real analytic in $\theta\in \T_\s^n$ for some
$\s>0$ (i.e. real when $(\csi,\heta)$ is real and $\theta\in\T^n$)
\\
Then there exists $\eps_*>0$ and $C>0$, such that for $|\eps_0|<\eps_*$:\\
- there
exists a set $\mathcal E_{\eps_0}\subset (0,2\pi]^n$ 
with $\meas ( (0,2\pi]^n\setminus \mathcal E_{\eps_0})\leq C \eps_0^{\frac 19}$, \\
- for any $\om\in\mathcal E_{\eps_0}$, there exists an analytic map $\theta\mapsto
 A_{\om}(\theta) \in {\rm sp}(2) $,  
 such that the change of coordinates
 \begin{equation}
   \label{kamclach}
(\csi',\heta')=e^{A_{\om}(\omega t)}(\csi,\heta) 
 \end{equation}
conjugates the Hamiltonian equations of \eqref{ham.cla} to the Hamiltonian equations of a homogeneous polynomial 
$$h_\infty(\csi, \heta)= \nu_1(\eps_0) \csi_1\heta_1 + c(\eps_0)  \csi_2^2$$ 
with
\be\label{nuc0}|\nu_1(\eps_0)-\nu_1|\leq C\eps_0, \quad
|c(\eps_0)-c_0|\leq \eps_0.
\ee
 Finally  $A_{\om}$
is $\eps_0$-close to zero and $e^{A_{\om}(\omega t)} $ leaves invariant
the space of real states.
\end{theorem}

The second reducibility theorem deals with 
Hamiltonians of the form (in the original variables \eqref{inivar})
\be \label{ham.cla.2}h_\eps(t,\csi, \heta)= \nu_1 \csi_1\heta_1
+ \nu_2 \csi_2\heta_2+\eps^3 q(\om
t,\csi_1,\heta_1,\csi_2,\heta_2)\ee
in which one has that the frequencies depend on $\omega\in\D$ and $\eps$ and
fulfill 
\begin{equation}
  \label{nu2}
c\leq |\nu_1(\omega)|\leq C,\quad c \eps^2 \leq |\nu_2(\omega)| \leq C \eps^2 , \qquad  \left|\partial_{\omega}\nu_1\right|,\left|\partial_{\omega}\nu_2\right|\leq
C\eps^2\ ,
  \end{equation}
with some positive $c, C$. So it is designed to deal with Hamiltonian \eqref{hs2}. The difference with the standard KAM context is that the second frequency is of order $\eps^2$ while the perturbation is of order $\eps^3$. The following theorem says that the standard conclusion still holds true, i.e. that the inhomogeneous Hamiltonian system associated with \eqref{ham.cla.2} is reducible for almost all values of $\omega$:
\begin{theorem}
\label{KAMclassico1} Assume \eqref{nu2}  and that $
\T^n\times\C^{4}\ni(\theta,\csi,\heta)\mapsto q(\theta,\csi,\heta)\in\C$
is a polynomial in $(\csi_1,\heta_1,\csi_2,\heta_2)$ homogeneous of degree 2, with coefficients analytic in $\theta\in
\T_\s^n$ for some $\s>0$ and taking real values when
$\theta\in\T^n$ and $\heta$ is the complex conjugate of $\csi$.
\\ Then there exists $\eps_*>0$ and $C>0$, such that for
$|\eps|<\eps_*$:\\ - there exists a set $\mathcal E_\eps\subset
(0,2\pi]^n$ with $\meas ( (0,2\pi]^n\setminus \mathcal E_\eps)\leq C
    \eps^{\frac 19}$, \\ - for any $\om\in\mathcal E_\eps$, there
    exists an analytic map $\theta\mapsto A_{\om}(\theta) \in {\rm
      sp}(2) $ such that the change of coordinates
 \begin{equation}
   \label{kamclach.1}
(\csi',\heta')=e^{A_{\om}(\omega t)}(\csi,\heta) 
 \end{equation}
conjugates the Hamiltonian equations of \eqref{ham.cla} to the Hamiltonian equations of a homogeneous polynomial 
$$h_\infty(\csi, \heta)= \tilde \nu_1(\eps) \csi_1\heta_1 + \tilde \nu_2(\eps)
\csi_2\heta_2 $$ 
with
\be\label{nuc}|\tilde \nu_j(\eps)-\nu_j|\leq C\eps^3, \quad j=1,2\ .\ee
 Finally  $A_{\om}$
is $\eps$-close to zero.
\end{theorem}

In the remainder of this section we give the details of the proof of Theorem
\ref{KAMclassico} while we only point out the small changes needed to
prove Theorem \ref{KAMclassico1}.

In fact, the proof of Theorem
\ref{KAMclassico} is very standard, but since
we're dealing with a degenerate Hamiltonian $h_0$, we have to be a
little careful. The fact that the perturbation $q$ is independent of 
$\heta_2$ is crucial here. If not, the Poisson bracket
$\{h,\chi\}$ could generate new quadratic terms in $(\csi_2,\heta_2)$ and the iteration could diverge.

\subsection{General strategy}
The canonical change of variables is constructed applying a KAM strategy to the  Hamiltonian 
$$h(y,\theta,\csi,\heta)= \om\cdot y+ \nu_0\csi_1\heta_1+\eps q(\om
t,\csi_1,\heta_1,\csi_2)$$
in the extended phase space $\R^n\times\T^n\times \C^{2}$ endowed with the standard symplectic form $dy\wedge d\theta+\im d\csi\wedge d\heta$. \\
We will say that the   Hamiltonian $h$ is in {\it normal form} if it reads
\be\label{h}
h(y,\theta,\csi,\heta)= \om\cdot y +a \csi_1\heta_1 + c \csi_2^2=\om\cdot y +N(\csi)\ee
where $a$ and $c$ are real constants (independent of $\theta$).

Let
 $q\equiv q_\om$ be a polynomial Hamiltonian homogeneous of degree 2. We write 
 \begin{equation}\label{q.form}
   \begin{aligned} q(\theta,\csi,\eta)&=\sum_{\alpha,\beta}q_{\alpha,\beta}(\theta) \csi^\alpha\heta^\beta
\end{aligned}
 \end{equation}
where the coefficients $q_{\alpha,\beta}(\theta)$ are analytic
functions of $\theta\in\T^n_\s$, we used the standard notation
$$
\csi^\alpha\heta^\beta:=\csi_1^{\alpha_1}\csi_2^{\alpha_2}\eta_1^{\beta_1}\heta_2^{\beta_2}
$$
and according to the independence on $\heta_2$ and the fact that this
must be a quadratic polynomial one has the restrictions
\begin{equation}
  \label{restri}
\beta_2=0\ ,\quad \left|\alpha\right|+|\beta|=2\ .
\end{equation}

The size of such polynomial function depending analytically on $\theta\in\T^n_\s$ and $C^1$ on $\om\in\D \subseteq (0,2\pi]^n$ will be controlled by the norm
 $$[q]_{\s,\D}:=\sup_{\substack{|\Im \theta|<\s,\ \omega \in
        \D \\\alpha,\beta,\ j=0,1}} |\p_\om^j q_{\alpha,\beta}(\theta)|$$
and we denote  by $\cQ(\s,\D)$ the corresponding class of Hamiltonians
of the form \eqref{q.form}-\eqref{restri} whose norm $[\cdot]_{\s,\D}$ is finite.
\\
Let us assume that $[q]_{\s,\D}=O (\eps)$. We search for $\chi\equiv \chi_\om\in  \cQ(\s,\D)$ with $\chi=O (\eps)$ such that its  time-one flow 
$\Phi_\chi\equiv \Phi_\chi^{t=1}$ transforms the Hamiltonian $h+ q$ into
\begin{equation}
\label{conj.app}
(h+ q(\theta))\circ \Phi_\chi=h_++ q_+(\theta), \qquad  \forall \omega \in \D_+ \ , 
\end{equation}
where $h_+=\om\cdot y + {N_+(\csi)}$ is a new normal form, $\eps$-close to $h$, and the new perturbation $q_+\in\mathcal Q(\s_+, {\D_+})$  is of size\footnote{Formally we could expect $q_+$ to be of size $O(\eps^2)$ but the small divisors and the change of analyticity domain will led to $O (\eps^{\frac 32})$. } $O (\eps^{\frac 32})$, 
{and $\D_+ \subset \D$ is an open set $\eps^\alpha$-close to $\D$ for some $\alpha >0$.}
As a consequence of the Hamiltonian structure we have  that
$$(h+ q(\theta))\circ \Phi_\chi= h+\{ h,\chi \}+q(\theta)+ O (\eps^{\frac 32}).$$
So to achieve the goal above 
we should  solve the {\it homological equation}:
\be \label{eq-hom}
\{ h,\chi \}= h_+-h -q(\theta)+O (\eps^{\frac 32}) , \quad {\omega \in \D_+} \ .
\ee
Repeating iteratively 
the same procedure with $h_+$ instead of $h$, we will construct a change of variable $\Phi$ such that
{
$$
(h+ q(\theta))\circ \Phi=h_\infty\,, \quad \omega \in \D_\infty
$$
with $h_\infty = \om\cdot y +N(\csi,\heta)$ in normal form and
$\D_\infty$ a $\eps^\alpha$-close subset of $\D$.}
 Note that  we will be forced to solve the homological equation, not
only for the original normal form $h_0=\om\cdot y +\nu \csi_1\heta_1$, but for 
more general normal form Hamiltonians
\eqref{h} with $N(\csi)=a \csi_1\heta_1 + c \csi_2^2$  close to $N_0=\nu \csi_1\heta_1$. To control this closeness we define a norm on $N=a \csi_1\heta_1 + c \csi_2^2$:
$$\| N\|:=\max(|a|,|c|).$$
The key remark is that $\forall q\in \mathcal Q(\sigma,\D)$ one has
$\left\{\csi_2;q\right\}\equiv 0$.

\subsection{Homological equation}\label{section-homo}

\begin{proposition}\label{prop:homo1}
Let $\D\subset\D_0$. Let $\D\ni\om\mapsto N(\om)$ be a $C^1$ mapping that  verifies 
\be\label{ass}
 \left\| \p_\om^j (N(\om)-N_0) \right\| \leq {\frac{\min(1,\nu_0)}{4}}
 \ee
for $j=0,1$ and $\om\in \D$.
Let $h=\om\cdot y+N(\csi)$, $q\in\mathcal Q(\s,\D)$ , $\ka>0$ and $K\ge 1$. \\
 Then there exists an open subset $\D'=\D'( \ka,K)\subset \D$, satisfying 
 \be\label{estim:D}\meas (\D\setminus \D')\leq  
 4d^2K^{2n}\ka, \ee
and there exist $\chi, r \in\cap_{0\leq\s'<\s}\mathcal Q(\s',\D')$ 
and $\tilde N$ in normal form such that for all $\om\in\D'$
\be\label{ho}
\{h,\chi \}+q=\tilde N+r\ .
\ee 
Furthermore {for all $0\leq \sigma' < \sigma$} 
 \begin{align}
\label{estim-homoR}
[r]_{\s',\D'}&\leq  C\ \frac{e^{-\frac12 (\s-\s')K}}{ (\s-\s')^{n}}
[q]_{\s,\D}\,,\\
\label{estim-homoS}
[\chi]_{\s',\D'}&\leq \frac{C}{\ka^2 (\s-\s')^{n}}
[q]_{\s,\D}\,,\\
 \label{estim-homoN}
 \|\p_\om^j \tilde N(\om)\|&\leq   [q]_{\s,\D}\quad j=0,1,\  \forall \om\in\D.
 \end{align} 
The constant $C$  depends on $n$.
\end{proposition}
\proof As usual we consider the "homological operator"
$\cL:=\left\{h;.\right\}$ and decompose the space $\mathcal Q(\sigma,\D)$ on the basis of
its eigenfunctions. Such a basis is given by the monomials
$$
\csi^\alpha\heta^\beta e^{\im k\theta}\ ,
$$
where $\alpha$ and $\beta$ are subject to the restrictions
\eqref{restri}. The corresponding eigenvalues are
\begin{equation}
  \label{eigen.1}
\im (\nu_1(\alpha_1-\beta_1)+\omega\cdot k)\ ,
\end{equation}
while Ker$(\cL)=span\left\{\csi_2^2,\csi_1\heta_1\right\}$. 
So, decomposing $q$ as in \eqref{q.form}, and expand the coefficients
in Fourier series:
$$
q_{\alpha,\beta}(\theta)=\sum_{k\in \Z^n}\hat
q_{\alpha,\beta}(k)e^{\im k\theta}\ ,
$$
then one is lead to define
\begin{equation}
  \label{chiab}
\chi_{\alpha,\beta}(\theta)=\sum_{|k|\leq K}\frac{\hat
q_{\alpha,\beta}(k)}{\im (\nu_1(\alpha_1-\beta_1)+\omega\cdot k)}e^{ik\theta}\ ,
\end{equation}
where, for $\alpha=(1,0),\beta=(1,0)$ and for $\alpha=(0,0)$ and
$\beta=(2,0)$ the sum is restricted to $k\not=0$.
We also define
\begin{align}
  \label{gli altri}
  \tilde N:=\hat q_{1,0,1,0}(0)\csi_1\heta_1+ \hat q_{0,2,0,0}(0)\csi_2^2\ ,
  \\
r(\csi,\heta,\theta):=\sum_{|k|>K,\alpha,\beta} \hat
q_{\alpha,\beta}(k)\csi^\alpha\heta^\beta e ^{ik\theta}\ ,  
\end{align}
so that the homological equation is satisfied. Still it remains to
prove the estimates of the various terms. We give explicitly the
estimate of $\chi$. To this end we have to control the small
denominators \eqref{eigen.1} under the restriction
\begin{equation}
  \label{restri.2}
\left|\alpha_1-\beta_1\right|+|k|\not =0\ ,\quad |k|\leq K\ ,\quad
\alpha_1+\beta_1\leq 2\ .
\end{equation}
We define $\D'$ to be the set for which the above small
denominators are bigger than $\kappa$. In order to estimate its
measure 
we  recall the following classical lemma:
\begin{lemma}\label{lem-mes}Let $f:[0,1]\mapsto\R$ be a $C^1$-map satisfying $|f'(x)|\geq \delta$ for all $x\in[0,1]$ and let $\ka>0$. Then
$$\meas\{x\in[0,1]\mid |f(x)|\leq \ka\}\leq \frac{\ka}{\delta}.$$
\end{lemma}
Since $|\partial_\om(k\cdot\om)(\frac k{|k|}))|=|k|\geq 1$,  we get, using   condition \eqref{ass}, 
\be\label{estimpd}|\partial_\om(k\cdot\om +\nu_1(\alpha_1-\beta_1))(\frac k{|k|}))|\geq 1/2\, .\ee
Using \eqref{estimpd} and  Lemma \ref{lem-mes}, we conclude that for
any fixed $k$
\be\label{inversebis}
 |k\cdot\om +\nu_1(\alpha_1-\beta_1)| > {\ka} \ ,  
 \ee
outside a set $F_{k,\alpha,\beta}$ of measure $\leq 2 d^2\ka$ {
  (the case $k=0$ being evident)} so that
if $F$ is the union of  $F_{k,\alpha,\beta}$ for $|k|< K$ and as
$\alpha$ and $\beta$ vary,  we have 
 \be\meas(F)\leq 4 K^{2n}d^2\ka \ . 
\ee
Thus, defining $\D'\equiv\D'( \ka,K)= \D\setminus F$, we get
{for all $\omega\in \D'$, $0 \leq \sigma' <\sigma$ and $\theta \in \T^n_{\sigma'}$}
 $$| \chi_{\alpha,\beta}(\theta,\om)|\leq \frac{C}{\ka (\s-\s')^{n}}
\sup_{|\Im \theta|<\s } |q_{\alpha,\beta}(\theta)|\,.$$
The estimates for the derivatives with respect to $\om$ are obtained
by differentiating the definition of $\chi$ (for more details see for instance \cite{Kuk2}).
\qed

\subsection{Iterative lemma}
 Theorem \ref{KAMclassico} is proved by an iterative KAM procedure. We begin with the initial Hamiltonian $h_0+q_0$ where
\be\label{h0}
h_0(y,\theta,\csi, \heta)= \om\cdot y +\nu_0 \csi_1\heta_1{+c_0\xi_2^2}\,,\ee
and $q_0={\eps_0 q}\in \cQ(\s_0,\D_0)$, {$\D_0 = [\eps,2\pi]^n$}. Then we construct iteratively the change of variables 
{$\Phi_{\chi_m}$}, the normal form $h_m=\om\cdot y +\nu_m \csi_1\heta_1+c_m\csi_2^2$ and
 the perturbation $q_m\in \cQ(\s_m,\D_m)$ 
  as follows: assume that the construction is done up to step $m\geq0$ then
\begin{itemize}
\item[(i)]
using Proposition \ref{prop:homo1}  we construct $\chi_{m+1}(\om,\theta)$, $\tilde N_{m}(\om)$,  $r_{m+1}(\om,\theta)$ and $\D_{m+1}\subset\D_m$ such that 
\be \label{eq-homo}
\{ h,\chi_{m+1} \}= \tilde N_{m} -q_m+r_{m+1},\quad \om\in\D_{m+1},\ \theta\in\T^n_{\s_{m+1}} 
\ee
where $0<\s_{m+1}<\s_m$ has to be chosen later; 
\item[(ii)] we define $h_{m+1}=\om\cdot y+ N_{m+1}$ by
\be\label{Nm} N_{m+1}=N_m+\tilde N_m\,,\ee
and
\be\label{qm}  q_{m+1}={r_{m+1}}  +\int_0^1\{(1-t)(h_{m+1}-h_m+r_{m+1})+tq_{m}, \chi_{m+1}\}\circ \Phi_{\chi_{m+1}}^t\dd t\,.
\ee
\end{itemize}
For any regular Hamiltonian $f$ we have, using the Taylor expansion of $g(t)=f\circ\Phi_{\chi_{m+1}}^t$ between $t=0$ and $t=1$
$$f\circ  \Phi_{\chi_{m+1}}=f+\{f,\chi_{m+1}\}+\int_0^1 (1-t)\{\{f,\chi_{m+1}\},\chi_{m+1}\}\circ \Phi_{\chi_{m+1}}^t\dd t\,.$$
Therefore we get for $\om\in\D_{m+1}$ 
\begin{align*}
(h_m+q_m)\circ  \Phi_{\chi_{m+1}}=h_{m+1}+q_{m+1}.
\end{align*}
Following the general scheme above we have for all $m\geq0$
$$(h_0+q_0)\circ \Phi^1_{\chi_{1}}\circ\cdots\circ \Phi^1_{\chi_m}= h_{m}+q_{m}.$$
At step $m$  the Fourier series are truncated at order $K_m$ and the small divisors are controlled by $\ka_m$. Now we specify the choice of all the parameters for $m\geq 0$ in term of $\eps_m$ which will control  $[q_m]_{\D_m,\s_m}$. \\
First we 
define  
 $\s_0=\s$ and for $m\geq 1$ we choose
\begin{align*}
\s_{m-1}-\s_m=&C_* \s_0 m^{-2},\\
K_m=&2(\s_{m-1}-\s_m)^{-1}\ln \eps_{m-1}^{-1},\\
\ka_{m}=&\eps_{m-1}^{\frac 18} \ , 
\end{align*}
where $(C_*)^{-1} =2\sum_{j\geq 1}\frac 1{j^2}$.\\

\begin{lemma}\label{iterative} There exists  $\eps_*>0$ depending on  $d$, $n$ such that, for  $|\eps_0|\leq\eps_*$ and
$$
\eps_{m}= \eps_0^{(3/2)^m} \ , \quad m\geq 0\,,
$$
 we have the following:\\
For all $m\geq 1$ there exist an open set $\D_m\subset\D_{m-1}$  functions 
$\chi_m,q_m\in \mathcal Q(\D_m,\s_m)$ and $N_m$ in normal form such that 
\begin{itemize}
\item[(i)]  The mapping \be \label{Phik} \Phi_{m}(\cdot,\om,\theta)=\Phi^1_{\chi_{m}}\ :\ \C^{2} \to \C^2, \quad \omega\in \D_{m},\ \theta\in\T_{\s_m}\ee
is a linear isomorphism, $C^1$ in $\om\in\D_m$, analytic in $\theta\in\T^n_{\s_m}$, 
 linking the Hamiltonian at step $m-1$ and the Hamiltonian at step  m, i.e.
$$(h_{m-1}+q_{m-1})\circ \Phi_{m}= h_m+q_m \ , \quad \forall \omega \in \D_m \ .$$
\item[(ii)] we have the estimates
\begin{align}
\label{DD}\meas(\D_{m-1}\setminus \D_{m})&\leq \eps_{m-1}^{\frac 19},\\
\label{NN} \|\p_\om^j {(N_m(\om) - N_{m-1}(\om) )} \|
&\leq \eps_{m-1},\ \quad j=0,1,\ \om\in\D_m,\\
\label{QQ}[q_m]_{\s_m,\D_m}&\leq \eps_m,\\
\label{FiFi}\| \Phi_m(\cdot,\om,\theta)-Id\|_{\mathcal L(\C^{2})}&\leq  \eps_{m-1}^{\frac 12},\ \text{ for } \theta\in \T_{\s_m}^n,\ \om\in\D_m.
\end{align}
\end{itemize}
 \end{lemma}
 \proof At step 1, $h_0=\om\cdot y +\nu_0\csi_1\heta_1$ and thus hypothesis \eqref{ass} is trivially satisfied and we can apply Proposition \ref{prop:homo1} to construct $\chi_1$, $N_1$, $r_1$ and $\D_1$ such that for $\om\in\D_1$
 $$\{h_0,\chi_1 \}+q_0=N_1-N_0 +r_1\ .$$
Then, using \eqref{estim:D}, we have
$$\meas(\D\setminus\D_1)\leq C K_1^{2n}\ka_1\leq  \eps_0^{\frac 19}$$
for $\eps=\eps_0$ small enough. 
Using \eqref{estim-homoS} we have for $\eps_0$ small enough
$$
[\chi_1]_{\s_1,\D_1}\leq C \frac{1}{\ka_1^{2}(\s_0-\s_1)^n}\eps_0\leq  \eps_0^{\frac12}.
$$
Similarly using \eqref{estim-homoR}, \eqref{estim-homoN} we have
$$\|N_1-N_0\|\leq  \eps_0,$$ and
$$
[r_1]_{\s_1, \D_1}\leq  C\frac{\eps_0^{2}}{(\s_0-\s_1)^n}\leq\eps_0^{\frac {3}2}
$$
for $\eps=\eps_0$ small enough.\\
In particular we deduce $\| \Phi_1(\cdot,\om,\theta)-Id\|_{\mathcal L(\C^{2})}\leq  \eps_0^{\frac 12}.$
 Thus using \eqref{qm} we get
for $\eps_0$ small enough
$$[q_1]_{\s_1,\D_1}\leq \eps_0^{3/2}=\eps_1.$$
 
  \medskip
 
 Now assume that we have verified Lemma \ref{iterative}  up to step $m$.
We want to perform the step $m+1$. We have $h_m=\om\cdot y +N_m$ and since 
$$ \|N_m-N_{0}\|\leq \|N_m-{N_{m-1}}\|+\cdots+ \|N_1-N_0\|\leq \sum_{j=0}^{m-1}\eps_j\leq 2\eps_0,$$
hypothesis \eqref{ass} is satisfied and we can apply Proposition \ref{prop:homo1}  to construct $\D_{m+1}$, $\chi_{m+1}$ and $q_{m+1}$. Estimates \eqref{DD}-\eqref{FiFi} at step $m+1$  are proved as we have proved the corresponding estimates  at step 1.
\endproof

\subsection{Transition to the limit and proof of Theorem \ref{KAMclassico}}

Let $${\cE_\eps}:=\cap_{m\geq 0}\D_m.$$ In view of \eqref{DD}, this is a Borel set satisfying
$$\meas(\D\setminus{\cE_\eps})\leq \sum_{m\geq 0} \eps_m ^{\frac 19}\leq 2 \eps_0^{\frac 19}.$$
Let us denote $\Psi_N(\cdot,\om,\theta)=\Phi_{1}(\cdot,\om,\theta)\circ\cdots\circ \Phi_N(\cdot,\om,\theta)$. Due  to \eqref{FiFi} it satisfies for $M\leq N$ and for $\om\in {\cE_\eps}$, $\theta\in\T_{\s/2}^n$
$$\| \Psi_N(\cdot,\om,\theta)- \Psi_M(\cdot,\om,\theta)\|_{\mathcal L(\C^{2})}\leq  \sum_{m=M}^N\eps_m^{\frac 12}\leq 2\eps_M^{\frac 12}\,.$$ 
Therefore  $(\Psi_N(\cdot,\om,\theta))_N$ is a Cauchy sequence in $\mathcal L(\C^{2})$. 
Thus when $N\to \infty$ the maps  $\Psi_N(\cdot,\om,\theta)$ converge to a limit mapping $\Psi_\infty(\cdot,\om,\theta)\in\cL(\C^{2}).$
 Furthermore since the convergence is uniform on $\om\in{\cE_\eps}$ and  $\theta\in\T_{\s/2}^n$, $(\om,\theta)\to\Psi_\infty^1(\cdot,\om,\theta)$ is analytic in $\theta$ and lipschitzian in  $\om$.
  Moreover, 
\be \label{estim-Phiinf}\|\Psi_\infty(\cdot,\om,\theta)-Id\|_{\mathcal L(\C^{2})} \leq \eps_{0}^{\frac 12}\,.\ee
By construction, the map $\Psi_m(\cdot,\om,\om t)$ transforms the original Hamiltonian 
$$
h_\eps(t,\csi, \heta)=N_0(\csi,\heta)+\eps q(\om t, \csi, \heta) , \quad N_0(\csi,\heta) = \nu_0 \csi_1\heta_1 
$$
 into
  $$
  H_m(t, \csi, \heta)=N_m(\csi,\heta)+q_m(\om t, \csi, \heta).
$$
When $m\to\infty$, by \eqref{QQ} we get $q_m\to 0$  and by \eqref{NN}  we get $N_m\to N$  where 
\begin{align}\label{Nom}
N\equiv N(\om) = N_{0} + \sum_{k=1}^{+\infty}\tilde{N}_{k}=:\nu(\om,\eps)\csi_1\heta_1+c(\om,\eps)(\csi_2-\heta_2).
\end{align}
Further for all $\om\in\cE_\eps$ we have using \eqref{NN}
$$\norma{N(\om)-N_0}\leq \sum_{m=0}^\infty \eps^m\leq 2 \eps. $$
Let us denote $\Psi_\infty(\theta)=\Psi_\infty^1(\cdot,\om,\theta)$. 
Denoting the limiting Hamiltonian 
{$h_\infty(\csi,\heta)= N_\infty(\csi,\heta)$} we have 
$$h_\eps(\theta,\Psi_\infty(\theta)(\csi, \heta)) = h_\infty(\csi, \heta),\quad \theta\in\T,\ (\csi, \heta)\in \C^{2d},\ \om\in\D_\eps\,.$$
Finally we show that the  linear symplectomorphism $\Psi_\infty$ can be written as \eqref{kamclach}.
To begin with,  write each Hamiltonian $\chi_m$ constructed in the KAM iteration as\begin{equation}
\label{matrix}
\chi_m(\theta, \csi, \heta)=\frac{1}{2} \left(\begin{matrix}
\csi \\ \heta
\end{matrix}\right)\cdot E_c \, B_m(\theta) \left(\begin{matrix}
\csi \\ \heta
\end{matrix}\right) \ , 
\quad
E_c:=\left[\begin{matrix} 
0 &  - \im   \\
\im  &  0 \end{matrix}\right] \  , 
\end{equation}
where   $B_m(\theta) $ is a skew-adjoint matrix of dimension $4\times 4$   of size  $\eps_m$. 
Then $\Psi_m$  has the form 
\begin{equation}
\label{ham.flow.r}
\Psi_m (\theta, \csi, \heta) = e^{ B_m(\theta)}(\csi, \heta)  \ .
\end{equation}
The following lemma is proved analogously to Lemma 3.5 in \cite{BGMR1}.
\begin{lemma}
\label{stitras}
There exists a sequence of  Hamiltonian matrices $A_l(\theta)$   such that 
\begin{equation}
\label{cano}
\Psi_{1}\circ...\circ \Psi_{l} (\theta, \csi, \heta) =e^{A_l(\theta)}(\csi,\heta)  \ \ \  \forall  \csi  \in\C^{2} \ . 
\end{equation} 
Furthermore, there exist
 a Hamiltonian matrix $A_\om(\theta)$    such  that 
   \begin{align}
\lim_{l\rightarrow+\infty}e^{A_l(\theta)} = e^{A_\infty(\theta)} \ , \ \ \ 
\sup_{|{\rm Im } \theta| \leq \sigma/2 } \Vert A_\om(\theta) \Vert \leq C\epsilon \ , 
\end{align}
and for each $\theta\in\T^n$,
$$\Psi(\theta, \csi, \heta)=e^{A_\om(\theta)}(\csi, \heta) \ \  \forall  \csi \in \C^2\ .$$
\end{lemma}

This concludes the proof of Theorem \ref{KAMclassico}.

\subsection{Changes for proving Theorem \ref{KAMclassico1}.}

The main change needed for the proof of Theorem \ref{KAMclassico1}
rests  in Proposition \ref{prop:homo1}. Indeed one has to
assume the r.h.s. of \eqref{ass} to be smaller than a small constant
times $\eps^2$ and the conclusion changes in the fact that at the
denominator of the r.h.s. \eqref{estim-homoS} instead of $\kappa^2$,
one gets
\begin{equation}
  \label{small.den}
\min\left\{\eps^2,\kappa^2\right\}\ .
\end{equation}
In the next lines we are going to prove this version of Proposition
\ref{prop:homo1}.

Indeed, the proof of Proposition \ref{prop:homo1} goes excatly in the
same way, except that now the eigenvalue \eqref{eigen.1} are
substituted by
\begin{equation}
  \label{eigen.2}
\im (\nu_1(\alpha_1-\beta_1)+\nu_2(\alpha_2-\beta_2)+\omega\cdot k)
\end{equation}
with the only selection rule
$$
\left|\alpha-\beta\right|+|k|\not=0\ ,\quad {\rm and}\ |k|\leq K\ .
$$
The estimates of the small divisors are obtained exactly in the same
way as far as
$$\left|\alpha_1-\beta_1\right|+|k|\not=0\ ,$$
however when such a quantity vanishes, the modulus of the small divisor becomes
$$
|2\nu_2|\geq c\eps^2
$$
(remark that in this case the normal form contains
span$\left\{\csi_1\heta_1\right\}$). Concerning contribution of the terms with this denominator to
the derivative with respect to $\omega$, remark that the involved
terms are
$$
\frac{\csi_2^2}{2\im \nu_2}\ ,\quad -\frac{\heta_2^2}{2\im \nu_2} \ 
$$
(multiplied by a constant). Thus the derivative with respect to, say
$\omega_i$ of the coefficient of one of  these terms gives, by \eqref{nu2},
$$
\left|-\frac{1}{2\nu_2^2}\frac{\partial\nu_2}{\partial \omega_i}
\right|\leq C\eps^{-2}\ ,
$$
which proves the claimed statement.

\vskip 5 pt

Then also the iterative lemma changes. Actually, only  few first steps 
change,  for which one has 
at the first step 
$$
\min\left\{\eps^2,\kappa_m^2\right\}=\eps^2\ ;
$$
it is easy to see that with the choices just before the statement of
Lemma  \ref{iterative}, after a finite number of steps  one has
$\min\left\{\eps^2,\kappa_m^2\right\}=\kappa_m^2$, and therefore  after this  step the
iteration can be repeated exactly in the same way as in the previous case.
\qed


\vspace{1em}

\footnotesize
\textbf{Acknowledgments:}
D. Bambusi and A. Maspero are  supported by~ PRIN 2020 (2020XB3EFL001)
“Hamiltonian and dispersive PDEs”, PRIN 2022 (2022HSSYPN)   "TESEO -
Turbulent Effects vs Stability in Equations from
Oceanography". D. Bambusi was also supported by GNFM.  
A. Maspero is supported by the European Union ERC CONSOLIDATOR GRANT 2023 GUnDHam,
Project Number: 101124921 and by GNAMPA.
D. Robert and B. Gr\'ebert benefited from the support of the Centre Henri Lebesgue ANR-11-LABX-0020-0 and  B. Gr\'ebert was partially supported by the ANR project KEN ANR-22-CE40-0016. C. Villegas-Blas partially supported by projects CONACYT Ciencia B\'asica  CB-2016-283531-F-0363 and PAPIIT-UNAM  IN116323.


\vspace{-2em}
\end{document}